\documentclass[letterpaper,10pt,conference]{ieeeconf}
\IEEEoverridecommandlockouts
\makeatletter

\let\proof\@undefined
\let\endproof\@undefined
\makeatother

\usepackage[bookmarks=true]{hyperref}
\usepackage{graphicx,psfrag}
\usepackage[noend]{algorithmic}
\usepackage{xspace}
\usepackage{amssymb,amsmath,amscd,mathrsfs}
\usepackage{amsthm}
\usepackage{comment,color}
\usepackage{subfigure}
\usepackage{url}
\usepackage{multirow}
\usepackage{etoolbox}

\newtoggle{proofs}
\togglefalse{proofs}

\usepackage[bookmarks=true]{hyperref}
\usepackage{fixltx2e}

\usepackage{color}
\usepackage{amsmath}
\usepackage{amssymb}
\usepackage{graphicx}
\usepackage{comment,xspace}
\usepackage{fancybox}

\newcommand{\real}{{\mathbb{R}}}
\newcommand{\reals}{\real}

\newtheorem{theorem}{Theorem}[section]

\newtheorem{lemma}[theorem]{Lemma}
\newtheorem{remark}[theorem]{Remark}

\newtheorem{definition}[theorem]{Definition}
\newtheorem{corollary}[theorem]{Corollary}

\newcommand{\prob}[1]{\mbox{$\mathbb{P}\left[#1\right]$}} 
\newcommand{\expectation}[1]{\mbox{$\mathbb{E}\left[#1\right]$}}

\newcommand{\fil}{\mathcal F}
\newcommand{\cs}{\mathcal Z}

\newcommand{\risk}{\rho}

\renewcommand{\baselinestretch}{0.9}

\newtoggle{paper}

\togglefalse{paper}

\title{\LARGE \bf
A Time Consistent Formulation of \\Risk Constrained Stochastic Optimal Control
}

\author{Yin-Lam Chow, Marco Pavone
\thanks{Y.-L. Chow and M. Pavone are with the Department of Aeronautics and Astronautics, Stanford University, Stanford, CA 94305, USA. Email: {\tt \{ychow, pavone\}@stanford.edu}.}}

\begin{document}

\maketitle
\thispagestyle{empty}
\pagestyle{empty}

\begin{abstract}
 Time-consistency is an essential requirement in risk sensitive optimal control problems to make rational decisions.  An optimization problem is time consistent if its solution policy does not depend on the time sequence of solving the optimization problem. On the other hand, a dynamic risk measure is time consistent if a certain outcome is considered less risky in the future implies this outcome is also less risky at current stage.
  In this paper, we study time-consistency of risk constrained problem where the risk metric is time consistent. From the Bellman optimality condition in \cite{Chow_Pavone_13_1}, we establish an analytical ``risk-to-go" that results in a time consistent optimal policy.  Finally we demonstrate the effectiveness of the analytical solution by solving Haviv's counter-example  \cite{haviv1996constrained} in time inconsistent planning.
\end{abstract}
\section{Introduction}
Stochastic Optimal Control (SOC) is concerned with sequential decision-making
under uncertainty. Consider a dynamical process that can be influenced by exogenous
noises as well as decisions made at every time step. 
The decision maker wants to optimize the behavior of the dynamical system over a certain time horizon by finding a policy that maps the history of states to optimal actions. 

In this SOC setup, we solve an optimization problem at present time and determine an optimal policy that maps states to actions subsequently to minimize cumulative cost. 
In order to obtain rational decisions, we aim to solve for a \emph{time consistent control policy} for which this solution policy is optimal to both the SOC at present and the tail-subproblems in all subsequent time steps. If such policy exists, the SOC problem is also known to be time consistent.
The property of time consistent problem is  formally stated as follows. The decision maker formulates
an optimization problem at time $k=0$ that yields a sequence of optimal decision
rules for $k=0$ to $k=N$. Then, at the next time
step $k=1$, he/she formulates a new problem starting at $k=1$ that yields a new sequence of
optimal decision rules from time steps $k=1$ to $N$. The sequence of policy is \emph{time
consistent} if the strategies obtained when solving the original problem at
time $k=0$ remain optimal for all subsequent problems. 

Recently, the concept of time-consistency has also been extended to the context of risk measures \cite{artzner2007coherent, Acciaio_11, Roorda:05, hardy2004iterated,
Riedel_03,Boda_06, Follmer_10}. In these papers, the authors formally defined the notion of time-consistency of risks, provided examples of time consistent risk measures and axiomatically justified that this property is necessary to develop rational risk assessments in stochastic processes.
In \cite{shapiro_12}, the author showed that expectation and worst case risk are the only time consistent coherent risks, and the authors in \cite{rus_09} developed a necessary and sufficient condition for time consistent risk measures.  Furthermore in \cite{Iancu_11}, the authors provided tight approximations of time inconsistent coherent risks by lower or upper-bounding them with time consistent metrics.

Note that common examples of time consistent risk measures include expectation and entropic risk measures. From \cite{rus_09}, by posing an unconstrained stochastic optimal control problem with a time consistent risk measure, one obtains a time consistent solution policy by dynamic programming. However \cite{Osogami_12} shows that a risk constrained problem is not necessarily time consistent even if  both objective function and constraints are time consistent risk measures. This results in undesirable outcomes
for example when a decision maker seeks to minimize expected loss subjected to a risk constraint, the optimal policy at present may become infeasible in future when the risk is re-evaluated. 

While time consistent policies are essential to ensure rational decisions, it has been pointed
out in \cite{haviv1996constrained, ross1989markov, sennott1995another} that in a constrained SOC setup, time consistency is not necessarily
satisfied by an optimal policy. There are several sufficient conditions to guarantee  time consistency for specific constrained SOC problems. In \cite{strotz1955myopia}, the author provided a sufficient condition for time consistency in deterministic optimal control problems. Also, \cite{Osogami_12} showed that the risk constrained problem is time consistent if the risk measures are \emph{optimality consistency}, i.e., any constraints that are feasible at present will also be also feasible in future. Furthermore \cite{haviv1996constrained,  ross1989markov} argued that a SOC problem is time consistent if constraints are satisfied at every sample history path. However the above sufficient conditions are either restricted to a small SOC problem  subclass or verifying this condition requires exponential computational complexity. 
In contrary the results in this paper shed
light to a simple and analytic sufficient condition of time-consistency for a general class of risk constrained SOC problems.

Our contributions of this paper are three-fold.
\begin{itemize}
\item First, we formulate a risk constrained SOC problem with time consistent risk measures and show that it can be solved by dynamic programming techniques (in the augmented action space). 
\item Second, by reformulating the above problem into an augmented Markov decision problem, we develop an analytical solution for the ``risk-to-go" that results a time consistent optimal control policy.
\item Third, we illustrate the effectiveness of this analytical method by solving for a time consistent optimal policy to Haviv's ``squander or save" counter-example \cite{haviv1996constrained} on time-inconsistent planning.
\end{itemize}

The rest of the paper is organized as follows. In Section \ref{sec:prelim} we provide a review of the theory of risk metrics and Markov decision processes. In Section \ref{sec:conceptual} we provide an analytical solution to the risk-to-go update that yields a time consistent optimal control policy. In Section \ref{sec:example} we further justify our solution method by solving for a time consistent optimal policy to the ``squander or save" problem. Finally, the conclusion and future work are discussed in Section \ref{sec:conc}.  

\section{Preliminaries}\label{sec:prelim}
In this section we provide some background for dynamic, time-consistent risk metrics and risk constrained SOC problems, on which we will rely extensively later in the paper.

\subsection{Markov Decision Processes}
A finite Markov Decision Process (MDP) is a four-tuple $(S, U, Q, U(\cdot))$, where $S$, the state space, is a finite set; $U$, the control space, is a finite set; for every $x\in S$, $U(x)\subseteq U$ is a nonempty set which represents the set of admissible controls when the system state is $x$; and, finally, $Q(\cdot|x,u)$ (the transition probability) is a conditional probability on $S$ given the set of admissible state-control pairs, i.e., the sets of pairs $(x,u)$ where $x\in S$ and $u\in U(x)$.

Define the space $H_k$ of admissible histories up to time $k$ by $H_k = H_{k-1} \times S\times U$, for $k\geq 1$, and $H_0=S$. A generic element $h_{0,k}\in H_k$ is of the form $h_{0,k} = (x_0, u_0, \ldots , x_{k-1}, u_{k-1}, x_k)$. Let $\Pi$ be the set of all deterministic policies with the property that at each time $k$ the control is a function of $h_{0,k}$. In other words, $\Pi := \Bigl \{ \{\pi_0: H_0 \rightarrow U,\, \pi_1: H_1 \rightarrow U, \ldots\} | \pi_k(h_{0,k}) \in U(x_k) \text{ for all } h_{0,k}\in H_k, \, k\geq 0 \Bigr\}$.

\subsection{Dynamic, Time-consistent, Risk Measures}
Consider a probability space $(\Omega, \fil, P)$, a filtration $\fil_1\subset \fil_2 \cdots \subset \fil_N \subset \fil$, and an adapted sequence of random variables $Z_k$, $k\in \{0, \cdots,N\}$. We assume that $\fil_0 = \{\Omega, \emptyset\}$, i.e., $Z_0$ is deterministic. In this paper we interpret the variables $Z_k$ as stage-wise costs. For each $k\in\{1, \cdots, N\}$, define the space of random variables with finite $p$th order moment as $\cs_k:=  L_p(\Omega, \fil_k, P)$, $p\in [1,\infty]$; also, let $\cs_{k, N}:=\cs_k \times \cdots \times \cs_N$. 

A dynamic risk measure is a sequence of monotone mappings  $\risk_{k,N}:\cs_{k, N}\rightarrow\cs_k$, $k\in\{0, \ldots,N\}$. Roughly speaking, a dynamic risk measure is \emph{time consistent} if it is such that, when a $Z$ cost sequence is deemed less risky than a $W$ cost sequence from the perspective of a future time $k$, and both sequences yield identical costs from the current time $l$ to the future time $k$, then the $Z$ sequence is deemed less risky at the current time $l$, as well. We refer to  \cite{rus_09} for a formal definition of time consistency. It turns out that dynamic, time-consistent risk metrics can be constructed by ``compounding" coherent one-step conditional risk measures, which are defined as follows.

\begin{definition}[Coherent One-step Conditional Risk Measures]
A coherent one-step conditional risk measures is a mapping $\risk_k:\cs_{k+1}\rightarrow \cs_k$, $k\in\{0,\ldots,N\}$, with the following four properties:
\begin{itemize}
\item Convexity: $\risk_k(\lambda Z + (1-\lambda)W)\leq \lambda\risk_k(Z) + (1-\lambda)\risk_k(W)$, $\forall \lambda\in[0,1]$ and $Z,W \in\cs_{k+1}$;
\item Monotonicity:  if $Z\leq W$ then $\risk_k(Z)\leq\risk_k(W)$, $\forall Z,W \in\cs_{k+1}$;
\item Translation invariance:  $\risk_k(Z+W)=Z + \risk_k(W)$, $\forall Z\in\cs_k$ and $W \in \cs_{k+1}$;
\item Positive homogeneity: $\risk_k(\lambda Z) = \lambda \risk_k(Z)$, $\forall Z \in \cs_{k+1}$ and $\lambda\geq 0$.
\end{itemize}
\end{definition} 

The compositional structure of dynamic, time-consistent risk metrics is then characterized by the following theorem.
\begin{theorem}[Dynamic, Time-consistent Risk Metrics  \cite{rus_09}]\label{thrm:tcc}
Consider, for each $k\in\{0,\cdots,N\}$, the mappings $\risk_{k,N}:\cs_{k, N}\rightarrow\cs_k$ defined as
\begin{equation}\label{eq:tcrisk}
\begin{split}
\risk_{k,N} &= Z_k + \risk_k(Z_{k+1} + \risk_{k+1}(Z_{k+2}+\ldots+\\
	&\qquad\risk_{N-2}(Z_{N-1}+\risk_{N-1}(Z_N))\ldots)),
\end{split}
\end{equation}
where the $\risk_k$'s are coherent one-step conditional risk measures. Then, the ensemble of such mappings is a dynamic, time-consistent  risk measure.
\end{theorem}

In this paper we consider a (slight) refinement of the concept of dynamic, time-consistent risk metric, which involves the addition of a Markovian structure \cite{rus_09} and enables the development of dynamic programming equations. 
\begin{definition}[Markov Dynamic Risk Measures \cite{rus_09}]\label{def:Markov}
Let $\mathcal V:=L_p(S, \mathcal B, P)$ be the space of random variables on $S$ with finite $p$th moment. Given a controlled Markov process $\{x_k\}$, a dynamic, time-consistent risk metric is a Markov dynamic risk metric if each coherent one-step conditional risk measure $\risk_k:\cs_{k+1}\rightarrow \cs_k$ in \eqref{eq:tcrisk} can be written as:
\begin{equation}\label{eq:Markov}
\risk_k(V(x_{k+1})) = \sigma_k(V(x_{k+1}),x_k, Q(x_{k+1} |x_k, u_k)),
\end{equation}
for all $V(x_{k+1})\in \mathcal V$ and $u\in U(x_k)$, where $\sigma_k$ is a coherent one-step conditional risk measure on $\mathcal V$ (with the additional technical property that for every $V(x_{k+1})\in \mathcal V$ and $u\in U(x_k)$ the function $x_k \mapsto \sigma_k(V(x_{k+1}), x_k, Q(x_{k+1}|x_k, u_k))$ is an element of $\mathcal V$). 
\end{definition}
In other words, in a Markov dynamic risk measures, the evaluation of risk is \emph{not} allowed to depend on the whole past.

\subsection{Stochastic Optimal Control with Dynamic, Time-consistent Risk Constraints}

Consider an MDP and let $c : S \times U \rightarrow \reals$ and $d : S \times U \rightarrow \reals$ be functions which denote costs associated with state-action pairs. Given a policy $\pi\in \Pi$, an initial state $x_0\in S$, and an horizon $N\geq 1$, the multi-stage cost function is defined as
\[
J^{\pi}_N(x_0):=\expectation{\sum_{k=0}^{N-1}\, c(x_k, u_k)},
\]
and the risk constraint is defined as
\[
R^{\pi}_N(x_0):= \risk_{0,N}\Bigl(d(x_0,u_0), \ldots, d(x_{N-1},u_{N-1}),0\Bigr),
\]
where $\risk_{k,N}(\cdot)$, $k\in \{0,\ldots, N-1\}$, is a Markov dynamic risk metric (for simplicity, we do not consider terminal costs, even though their inclusion is straightforward).  The problem is then as follows:
\begin{quote} {\bf Optimization problem $\mathcal{OPT}$} --- Given an initial state $x_0\in S$, a time horizon $N\geq 1$, and a risk threshold $r_0 \in \reals$, solve
\begin{alignat*}{2}
\min_{\pi \in \Pi} & & \quad&J^{\pi}_N(x_0) \\  
\text{subject to} & & \quad&R^{\pi}_N(x_0) \leq r_0.
\end{alignat*}
\end{quote}
If problem $\mathcal OPT$ is not feasible, we say that its value is $\infty$. In \cite{Chow_Pavone_13_1} the authors developed a dynamic programing approach to solve this problem. To define the value functions, one needs to define the tail subproblems. For a given $k\in \{0,\ldots,N-1\}$ and a given state $x_k\in S$, we define the \emph{sub-histories} as $h_{k,j}:=(x_k, u_k, \ldots,x_j)$ for $j\in \{k,\ldots, N\}$; also, we define the \emph{space of truncated policies} as $\Pi_k:=\Bigl \{ \{\pi_k, \pi_{k+1}, \ldots\} | \pi_j(h_{k,j}) \in U(x_j) \text{ for } j\geq k \Bigr\}$. For a given stage $k$ and state $x_k$, the cost of the tail process associated with a policy $\pi\in \Pi_k$ is simply $J^{\pi}_N(x_k):=\expectation{\sum_{j=k}^{N-1}\, c(x_j, u_j)}$. The risk associated with the tail process is:
\[
R^{\pi}_N(x_k):= \risk_{k,N}\Bigl(d(x_k,u_k), \ldots, d(x_{N-1},u_{N-1}),0\Bigr).
\]
The tail subproblems are then defined as
\begin{alignat}{2}
\min_{\pi \in \Pi_k} & & \quad&J^{\pi}_N(x_k) \label{problem_SOCP}\\
\text{subject to} & & \quad&R^{\pi}_N(x_k) \leq r_k(x_k),\label{constraint_SOCP}
\end{alignat}
for a given (undetermined) threshold value $r_k(x_k) \in \reals$ (i.e., the tail subproblems are specified up to a threshold value).

For each $k\in\{0,\ldots, N-1\}$ and $x_k\in S$, we define the set of feasible constraint thresholds as 
\[
\Phi_k(x_k):=[\underline{R}_N(x_k), \infty),\quad \Phi_N(x_N):=[0,\infty),
\]
where  $\underline{R}_N(x_k):=\min_{\pi\in \Pi_k} \, R_N^{\pi}(x_k)$. One then defines the value functions as follows:
\begin{itemize}
\item If $k<N$ and $r_k \in \Phi_k(x_k)$:
\begin{alignat*}{2}
V_k(x_k, r_k)  = &\min_{\pi \in \Pi_k} & \quad&J^{\pi}_N(x_k) \\  
&\text{subject to} & &R^{\pi}_N(x_k) \leq r_k.
\end{alignat*}
\item If $k\leq N$ and $r_k \notin \Phi_k(x_k)$:
\[
V_k(x_k, r_k)  = \infty.
\]
\item When $k=N$ and $r_N\in\Phi_N(x_N)=[0,\infty]$:
\[
V_N(x_N,r_N) = 0.
\]
\end{itemize}
Let $B(S)$ denote the space of real-valued bounded functions on $S$, and $B(S \times \reals)$ denote the space of real-valued bounded functions on $S\times \reals$. For $k\in \{0, \ldots, N-1\}$, we define the dynamic programming operator $T_k[V_{k+1}] : B(S \times \reals) \mapsto B(S \times \reals)$ according to the equation:

\begin{equation}\label{eq:T}
\begin{split}
T_k[V_{k+1}]&(x_k, r_k) := \inf_{(u,r^{\prime})\in F_k(x_k, r_k)} \, \biggl\{c(x_k,u) \, \,+\\
& \ \sum_{x_{k+1}  \in S} \, Q(x_{k+1}|x_k,u)\, V_{k+1}(x_{k+1}, r^{\prime}(x_{k+1})) \biggr\},
\end{split}
\end{equation}
where $F_k$ is the set of control/threshold \emph{functions}:
\begin{equation*}
\begin{split}
F_k(x_k,& r_k):= \biggr\{(u, r^{\prime}) \Big | u\in U(x_k), r^{\prime}(x^{\prime}) \in \Phi_{k+1}(x^{\prime}) \text{ for}\\& \text{all } x^{\prime} \in S, \text{ and } d(x_k, u) + \risk_k(r^{\prime}(x_{k+1}))  \leq r_k\biggl\}.
\end{split}
\end{equation*}
If $F_k(x_k,r_k) = \emptyset$, then $T_k[V_{k+1}](x_k, r_k)=\infty$.

For a given state and threshold constraint, $F_k$ characterizes the set of feasible pairs of actions and subsequent constraint
thresholds. Feasible subsequent constraint thresholds are thresholds which if satisfied at the next stage ensure that the current
state satisfies the given constraint threshold. Note that the value functions are defined on an \emph{augmented} state space, which combines the original (discrete) states $x_k$ with the \emph{real-valued risk-to-go} states $r_k$. We will refer to the MDP problem associated with such augmented state space as augmented MDP (AMDP).  The main result in  \cite{Chow_Pavone_13_1}  is the following theorem about the correctness of value iteration for AMDP.

\begin{theorem}[Bellman's Equation with Risk Constraints  \cite{Chow_Pavone_13_1}]\label{TC_good}
For all $k\in \{0, \ldots, N-1\}$ the value  functions satisfy the Bellman's equation:
\[
V_k(x_k, r_k) = T_k[V_{k+1}](x_k, r_k).
\]
\end{theorem}

Next, we present a procedure to construct optimal policies. Under the assumptions of Theorem \ref{TC_good}, for any given $x_k\in S$ and $r_k \in \Phi_k(x_k)$ (which implies that $F_k(x_k, r_k)$ is non-empty), let $u^*(x_k, r_k)$ and $r^{\prime}(x_k, r_k)(\cdot)$ be the minimizers in equation \eqref{eq:T}. Next theorem shows how to construct history dependent optimal policies.

\begin{theorem}[Optimal Policies]\label{them:optPoli}
Let $\pi \in \Pi$ be a policy recursively defined as:
\[
\pi_k(h_k) = u^*(x_k, r_k), \text{ with } r_k = r^{\prime}(x_{k-1},r_{k-1})(x_k),
\]
when $k\in \{1,\ldots, N-1\}$, and
\[
\pi(x_0) = u^*(x_0, r_0),
\]
for a given threshold $r_0\in \Phi_0(x_0)$. Then, $\pi$ is an optimal policy for problem $\mathcal{OPT}$ with initial condition $x_0$ and constraint threshold $r_0$.
\end{theorem}

Interestingly, if one views the constraint thresholds as state variables (whose dynamics are given in the statement of Theorem \ref{them:optPoli}), the optimal (history-dependant) policies of problem $\mathcal{OPT}$ have a Markovian structure with respect to the augmented control problem.

\section{Time consistency and conceptual risk-to-go}\label{sec:conceptual}
We start this section by demonstrating some time-inconsistency behaviors in risk-sensitive optimal control problems using several counter-examples.
\subsection{Time Inconsistent Planning Leads to Irrational Behaviors}
The most common strategy to model risk awareness in MDPs is to consider a static risk metric (i.e., a metric assessing risk from the perspective of a \emph{single} point in time) applied to the entire stream of future costs. Typical examples include variance-constrained MDPs \cite{Piunovskiy_06, Sniedovich_80, Mannor_11}, or problems with probability constraints \cite{Piunovskiy_06,  ono2008iterative, blackmore2010probabilistic, blackmore2009convex}, which are popular in the robotics community  (in these problems risk is assessed only from the perspective of the initial stage). However, since static risk metrics do not involve a reassessment of risk at subsequent decision stages they generally lead to \emph{irrational} behaviors. For example, a UV can seek to incur losses (i.e., dangerous maneuvers) or can deem as dangerous states that are indeed favorable under any realization of the underlying uncertainty. 

In this subsection, we will illustrate some irregular behaviors in risk sensitive multi-period planning by two examples.  
\begin{quote} {\bf Example 1: Variance-constrained planning} --- Given an MDP with initial state $x_0\in S$ and time horizon $N\geq 1$, solve
\begin{alignat*}{2}
\min_{\pi} & & \quad&\expectation{\sum_{k=0}^{N-1}\, c(x_k, u_k) + c_N(x_N)}\\
\text{subject to} & & \quad&\mathrm{var}\biggl(\, \sum_{k=0}^{N-1}\, d(x_k, u_k) +d_N(x_N) \, \biggr) \leq r_0,
\end{alignat*}
where $r_0 \in \reals$ is a user-provided risk threshold.
\end{quote}
Consider the example in Figure \ref{fig:vcp}. When the risk threshold $r_0$ is below 25, policy $\pi_1$ is infeasible and the optimal policy is $\pi_2$. According to policy $\pi_2$, if the decision maker does not incur a cost in the first stage it \emph{seeks to incur losses} in subsequent stages to keep the variance small. This can be seen as a consequence of the fact that Bellman's principle of optimality does not hold for this class of problems.

\begin{figure}[h]
  \centering
 \subfigure[Stage-wise constraint and objective function costs and transition probabilities  for policy $\pi_1$.]{
  \includegraphics[width =0.5\linewidth]{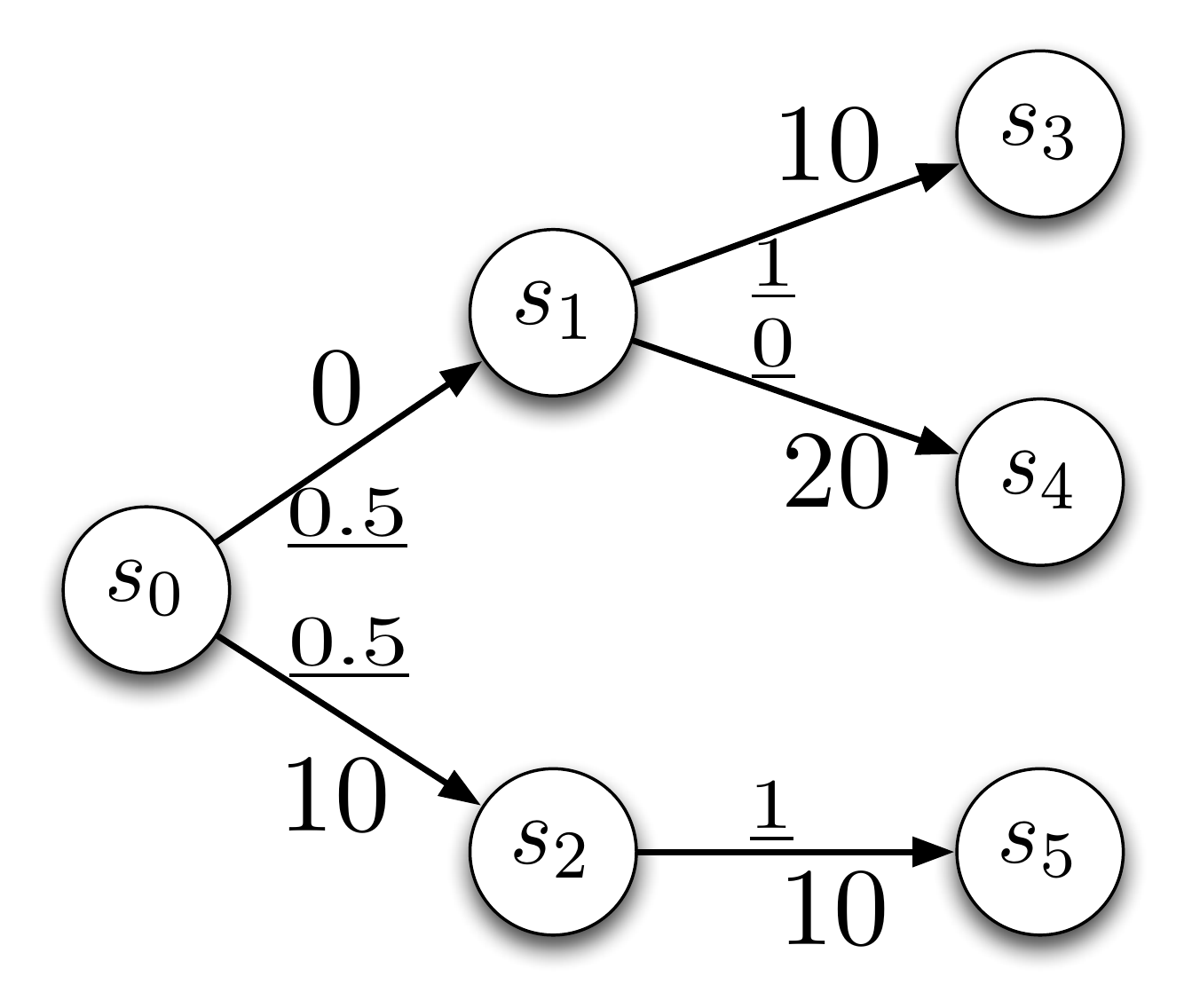}
  } \qquad \qquad \qquad
 \subfigure[Stage-wise constraint and objective function costs and transition probabilities  for policy $\pi_2$.]{
  \includegraphics[width =0.5\linewidth]{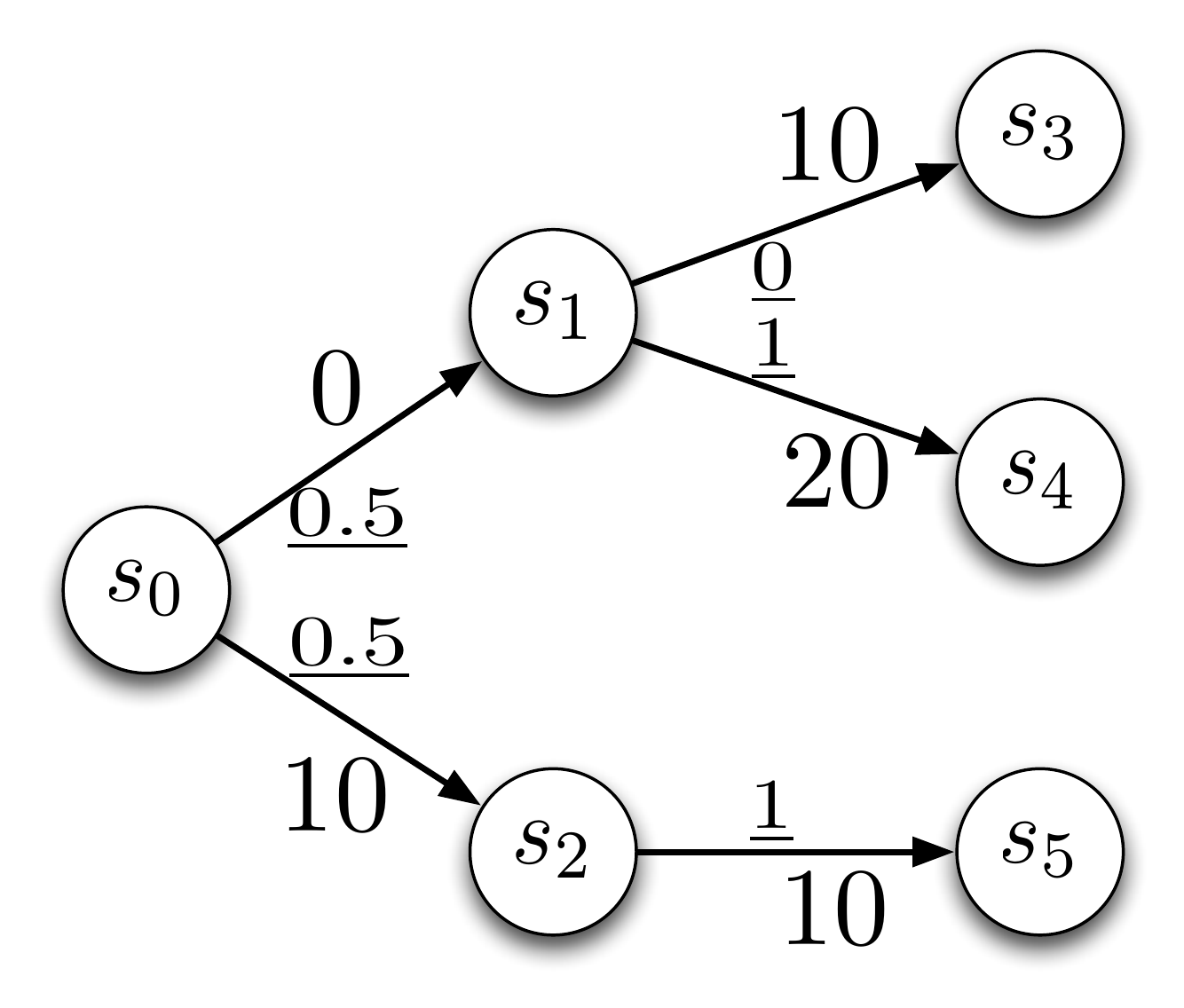}
  }
  \caption{Limitations of mean-variance optimization. Underlined numbers along the edges represent transition probabilities; non-underlined numbers represent stage-wise constraint and objective function costs (that are equal for this example). Terminal constraint costs are zero. Under policy $\pi_1$, the costs per stage are given by $d(s_0, u_0)= 0.5\cdot 0 + 0.5\cdot10 = 5$,  $d(s_1, u_1)  =10$, and $d(s_2, u_1)  =10$; under policy $\pi_2$, the costs per stage are given by $d(s_0, u_0)= 5$,  $d(s_1, u_1)  =20$, and $d(s_2, u_1)  =10$. One can verify that for policy $\pi_1$ one has  $\mathrm{var} \Bigl(\sum_{k=0}^{N -1}\, d(x_k, u_k) + d(x_N)\Bigr) = 25$, while for policy $\pi_2$ one has $\mathrm{var}\Bigl(\sum_{k=0}^{N-1} \, d(x_k, u_k) + d(x_N)\Bigr) = 0$. Then, if the risk threshold is less than $25$, the decision-maker would choose policy $\pi_2$ and would seek to incur losses in order to keep the variance small enough.}
  \label{fig:vcp}
\end{figure}

As a second example, we consider MDPs with average value at risk (AVaR) constraints, which are closely related to chance (i.e., probability) constraints and are enjoying a growing popularity, especially in the finance industry \cite{borkar2010risk}, due to favorable computational aspects, such as convexity. The average value at risk for a random variable $X$ at confidence level $\alpha$ is defined as \cite{Rockafellar_Uryasev_02}:
\[
\mathrm{AVaR}_{\alpha}(X):= \frac{1}{1 - \alpha}\, \int_{\alpha}^{1} \, \mathrm{VaR}_{\tau}(X) \, d\tau,
\]
where $\mathrm{VaR}_{\alpha}(X)$ is simply the $\alpha$-quantile of random variable $X$, i.e.,
\[
\mathrm{VaR}_{\alpha}(X) = \min \{x | \, \prob{X\leq x}\geq \alpha \}.
\]
Intuitively, the $\mathrm{AVaR}_{\alpha}$ is the expectation of $X$ in the conditional distribution of its upper $\alpha$-tail. For this reason, it can be interpreted as a metric of ``how bad is bad." The risk metric $\mathrm{AVaR}_{\alpha}$ is closely related to chance constraints, since the constraint $\mathrm{VaR}_{\alpha}(X)\leq 0$ corresponds to the chance constraint $\prob{X\leq 0}\geq \alpha$ \cite{Rockafellar_Uryasev_02}.
\begin{quote} {\bf Example 2: AVaR-constrained planning} --- Given an MDP with initial state $x_0\in S$ and time horizon $N\geq 1$, solve
\begin{alignat*}{2}
\min_{\pi} & & \quad&\expectation{\sum_{k=0}^{N-1}\, c(x_k, u_k) + c_N(x_N)} \\
\text{subject to} & & \quad&\mathrm{AVaR}_{\alpha}\biggl(\, \sum_{k=0}^{N-1}\, d(x_k, u_k) + d_N(x_N) \, \biggr) \leq r_0,
\end{alignat*}
where $r_0 \in \reals$ is a user-provided risk threshold.
\end{quote}

Let us interpret the constraint costs $d$ as acceptable if negative and unacceptable otherwise. Accordingly consider the example in Figure \ref{fig:avar} (based upon \cite{artzner_delbaen_eber_heath_98}), with threshold $r_0= 0$ and confidence level $1/3$. One can show that the problem (consisting of a single policy) is infeasible, since at the first stage AVaR is positive. On the other hand, the constraint costs are acceptable in every state of the world from the perspective of the subsequent stage. In other words, the decision-maker would deem infeasible a problem that, at the second stage, appears feasible under any possible realization of the uncertainties.

\begin{figure}[htpb]
\begin{center}
\includegraphics[width=0.45\textwidth]{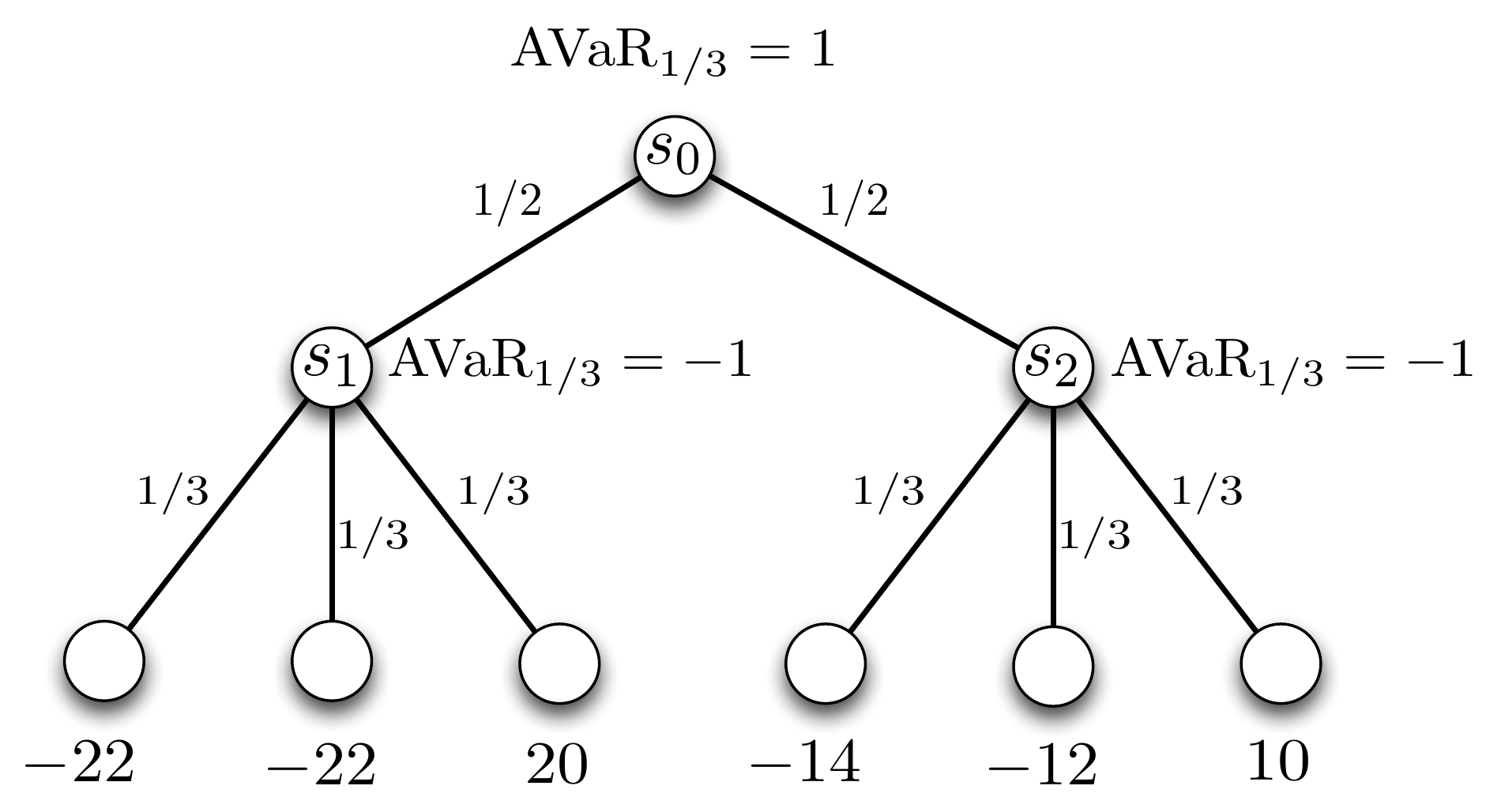}
\end{center}
\caption{Limitations of AVaR-constrained optimization. The numbers along the edges represent transition probabilities, while the numbers below the terminal nodes represent the terminal constraint costs (the other constraint costs are zero). The problem involves a single control policy (hence there is a unique transition graph). The constraint cost appears acceptable in states $s_1$ and $s_2$, but unacceptable from the perspective of the first stage in state $s_0$.}
\label{fig:avar}
\end{figure}

It is important to note that there is nothing special about these examples, which indeed capture a range of widely accepted criteria to pose risk constraints. Similar paradoxical results could be obtained with other risk metrics. Henceforth, we will collectively refer to the aforementioned irrational behaviors as ``time-inconsistent" policies, since they reflect an inconsistent risk assessment over time.
\subsection{Time Consistency }
From Theorem \ref{TC_good} and Theorem \ref{them:optPoli}  we can find a sequence of history dependent optimal control policies by Bellman iteration. In this section, we want to show that the risk constrained SOC in Problem $\mathcal{OPT}$ is time consistent. Most analysis in literature (c.f. Chapter~1 of \cite{bertsekas_05, chong2012bellman} for more details), restricts the analysis of time consistency to  problems with Markovian policies.  Since Problem $\mathcal{OPT}$ solves for a sequence of history dependent policy, it is unclear how we can analyze time consistency directly by truncating the sub-histories. For this purpose, we define the space of augmented state feedback policies and the space of risk-to-go updates:
\begin{alignat*}{1}
\Pi^{m}_k &:= \left \{ \{\pi_k,\ldots,\pi_{N-1}\}: 
\begin{array}{l}
\pi_j:S\times\Phi_j(S) \rightarrow U(S),\\
  j\in\{ k,\ldots,N-1\} \end{array}\right\} \\
\mathfrak{R}_k &\!\!:=\! \!\left \{ \{\mathcal{R}_{k+1},\ldots,\mathcal{R}_{N}\}\!:\!\!\!\!
\begin{array}{l}
 \mathcal{R}_{j+1}(x_{j},r_{j})(x_{j+1})=r_{j+1}-r_j,\\
 r_{j+1}\in\Phi_{j+1}(x_{j+1}),\\
  j\in\{ k,\ldots,N-1\} \end{array}
\!\!\!\!\right\}
\end{alignat*}
and the following optimization problem:
\begin{quote} {\bf Optimization problem $\mathcal{OPT}^\textrm{M}$} --- Given an initial state $x_0\in S$, a time horizon $N\geq 1$, and a risk threshold $r_0 \in \reals$, solve
\begin{alignat*}{2}
\min_{\pi \in \Pi^{m}_0, \mathcal{R}\in\mathfrak{R}_0} & & \quad&J^{\pi}_N(x_0)\\
\text{subject to} & & \quad&\risk_{j}(r_{j+1}+d_j(x_j,\pi_j(x_j,r_j)))=r_j,\\
& &\quad& r_{j+1}=\mathcal{R}_{j+1}(x_{j},r_{j})(x_{j+1})+r_j\\
& &\quad & 0\leq r_N, \quad j\in\{0,\ldots,N-1\}.
\end{alignat*}
\end{quote}

The $k-$th subproblem of Problem $\mathcal{OPT}^\textrm{M}$ is simply defined by replacing $x_0\in S$ and $r_0\in\Phi_0(x_0)$ by $x_k\in S$ and $r_k\in\Phi_k(x_k)$, i.e.,
\begin{itemize}
\item If $k<N$ and $r_k \in \Phi_k(x_k)$:
\begin{alignat*}{2}
V^M_k(x_k, r_k)  =& & \quad&\\
 \min_{\pi \in \Pi^m_k,\mathcal{R}\in\mathfrak{R}_k}& & \quad&J^{\pi}_N(x_k) \\  
\text{subject to} & & \quad&\risk_{j}(r_{j+1}+d_j(x_j,\pi_j(x_j,r_j)))=0,\\
& &\quad& r_{j+1}=\mathcal{R}_{j+1}(x_{j},r_{j})(x_{j+1})+r_j\\
& &\quad & 0\leq r_N, \quad j\in\{k,\ldots,N-1\}.
\end{alignat*}
\item If $k\leq N$ and $r_k \notin \Phi_k(x_k)$:
\[
V^M_k(x_k, r_k)  = \infty.
\]
\item When $k=N$ and $r_N\in\Phi_N(x_N)=[0,\infty]$:
\[
V^M_N(x_N,r_N) = 0.
\]
\end{itemize}
We are now in the position of defining the notion of time consistency for Problem $\mathcal{OPT}^M$.

\begin{definition}(Time Consistency of risk constrained SOC)\label{TC_defn}
For any given initial state $x_0\in S$, and risk threshold $r_{0}\in\Phi_0(x_0)$, define $\pi^\ast=\{\pi_0^\ast,\ldots,\pi_{N-1}^\ast\}\in\Pi^M_0$ as a sequence of optimal policy and $r^\ast=\{r_1^\ast,\ldots,r_{N}^\ast\}$ as a sequence of risk-to-go for Problem $\mathcal{OPT}^M$.  Problem $\mathcal{OPT}^M$ is a time consistent SOC problem, if at any $k\in\{0,\ldots,N-1\}$, the k-subsequence of $\pi^\ast$: $\pi^\ast_{k,N-1}:=\{\pi_k^\ast,\ldots,\pi_{N-1}^\ast\}$ and $r^\ast$: $r^\ast_{k+1,N}:=\{r_{k+1}^\ast,\ldots,r_{N}^\ast\}$ are sequences of optimal solution to the k-tail subproblem of Problem $\mathcal{OPT}^M$ at initial state $x_k\in S$, and risk threshold $r_{k}\in\Phi_k(x_k)$. In this case, $\pi^\ast$ is a time consistent optimal policy, and $\mathcal R^\ast$ is a time consistent risk-to-go.
\end{definition}

Before getting into the main result, we want to justify the equivalence between Problem $\mathcal{OPT}^M$ and $\mathcal{OPT}$. First, we have the following technical lemma, showing that without loss of optimality, the inequality constraint in $F(x_k,r_k)$ can be replaced by an equality.
\begin{lemma}\label{tech_lem_F}
For any value function $V_{k+1}:S\times \Phi_{k+1}(S)\rightarrow\reals$, $\forall k\in\{0,\ldots,N-1\}$ for Problem $\mathcal{OPT}$, the following equality holds:
\begin{equation}\label{eq:T_L}
V_k(x_k,r_k)=T_k[V_{k+1}](x_k,r_k)=\overline{T}_k[V_{k+1}](x_k,r_k),\,\forall x_k,\, r_k.
\end{equation}
where
\begin{equation*}
\begin{split}
\overline{T}_k[V]&(x_k, r_k) := \inf_{(u,r^{\prime})\in \overline{F}_k(x_k, r_k)} \, \biggl\{c(x_k,u) \, \,+\\
& \ \sum_{x_{k+1}  \in S} \, Q(x_{k+1}|x_k,u)\, V(x_{k+1}, r^{\prime}(x_{k+1})) \biggr\},
\end{split}
\end{equation*}
and $\overline{F}_k$ is the set of control/threshold \emph{functions}:
\begin{equation*}
\begin{split}
&\overline{F}_k(x_k, r_k):= \biggr\{(u, r^{\prime}) \Big | u\in U(x_k), r^{\prime}(x^{\prime}) \in \Phi_{k+1}(x^{\prime}) \text{ for}\\
& \text{all } x^{\prime} \in S, r^\prime(x^\prime)=r_k+L(x^\prime),\,\, \text{and }\,\,\risk_k(L)+d(x_k,u)=0\biggl\}.
\end{split}
\end{equation*}
\end{lemma}
\begin{proof}
First it is obvious that $\overline{F}_k(x_k, r_k)\subseteq {F}_k(x_k, r_k)$ and $T_k[V_{k+1}](x_k,r_k)\leq \overline{T}_k[V_{k+1}](x_k,r_k)$. Now, suppose $F_k(x_k,r_k)\neq\emptyset$ and there exists an risk-to-go $r^{\prime,\ast}(x_{k+1})$ such that 
\[
d(x_k, u^\ast) + \risk_k(r^{\prime,\ast}(x_{k+1})) < r_k
\]
where $u^\ast$ is the optimal control input solved from the Bellman's equation in (\ref{eq:T}). Recall the definition of the value function of Problem $\mathcal{OPT}$, $V_{k+1}(x_{k+1}, r_{k+1})=\min\{ J^{\pi}_N(x_{k+1}):\,\,\pi \in \Pi_{k+1},\,\,R^{\pi}_N(x_{k+1}) \leq r_{k+1}\}$ when $r_{k+1} \in \Phi_{k+1}(x_{k+1})$. It can be easily seen that $V_{k+1}(x_{k+1}, r_{k+1})$ is a non-increasing function with respect to $r_{k+1}$. Furthermore, since $\risk_k$ is Lipschitz and $r^{\prime,\ast}(x_{k+1})$ is a discrete state, continuous magnitude random variable, we can always find a nonnegative discrete state, continuous magnitude, bounded random variable $\epsilon(x_{k+1})$ such that
\[
d(x_k, u^\ast) + \risk_k(r^{\prime,\ast}(x_{k+1})+\epsilon(x_{k+1})) = r_k.
\]
Furthermore, we know that $r^{\prime,\ast}(x_{k+1})+\epsilon(x_{k+1})\geq r^{\prime,\ast}(x_{k+1})$ surely, which implies $r^{\prime,\ast}(x_{k+1})+\epsilon(x_{k+1})\geq \underline{R}_N(x_{k+1})$. On the other hand, since $r_k$, $\epsilon(x_{k+1})$ and $r^{\prime,\ast}(x_{k+1})$ are all bounded quantities, this implies $r^{\prime,\ast}(x_{k+1})+\epsilon(x_{k+1})<\infty$ surely. By writing $\tilde{r}^{\prime,\ast}(x_{k+1})=r^{\prime,\ast}(x_{k+1})+\epsilon(x_{k+1})\geq r^{\prime,\ast}(x_{k+1})$, one obtains $V_{k+1}(x_{k+1}, r^{\prime,\ast}(x_{k+1}))\geq V_{k+1}(x_{k+1}, \tilde{r}^{\prime,\ast}(x_{k+1}))$ surely. Thus, we can use $(u^\ast, \tilde{r}^{\prime,\ast})\in\overline{F}_k(x_k, r_k)$ as a minimizer for the Bellman's equation in (\ref{eq:T}). To summarize, whenever $F_k(x_k,r_k)\neq\emptyset$, there always exists an optimal control input $u^\ast$ and  risk-to-go $r^{\prime,\ast}(x_{k+1})$ for Bellman's equation in (\ref{eq:T}) such that $d(x_k, u^\ast) + \risk_k(r^{\prime,\ast}(x_{k+1})) = r_k$. By letting $L(x^\prime)=r^{\prime,\ast}(x^\prime)-r_k$, $\forall x^\prime\in S$, we have just showed that there is no loss of optimality to consider the operator $\overline{T}_k[V_{k+1}](x_k,r_k)$ instead of the Bellman's equation in (\ref{eq:T}).
\end{proof}

Essentially, we have the following result. The proof of this theorem is analogous to the proof of Theorem \ref{TC_good} and is omitted here for the interest of brevity. Details of the proof can be found in the Appendix.
\begin{theorem}\label{Bellman_M_thm}
The value function of Problem $\mathcal{OPT}^M$ is identical to the value function of Problem $\mathcal{OPT}$. Furthermore, Problem $\mathcal{OPT}^M$ is a time consistent SOC problem. 
\end{theorem}
 Assume the infimum in expression (\ref{eq:T_L}) is attained. Let $r_0^\ast=r_0$. Based on Definition \ref{TC_defn}, any optimal control polices and optimal  risk-to-go found from the sequence of Bellman's equation, $\forall k\in\{0,\ldots,N-1\}$: 
 \[
 \begin{split}
& (\pi_k^\ast(x_k,r^\ast_k) ,r_{k+1}^\ast)\in\text{argmin}_{(u,r^{\prime})\in \overline{F}_k(x_k, r^\ast_k)} \, \biggl\{\\
&c(x_k,u) \, \,+ \ \sum_{x_{k+1}  \in S} \, Q(x_{k+1}|x_k,u)\, V^M_{k+1}(x_{k+1}, r^{\prime}(x_{k+1})) \biggr\}
\end{split}
 \]
 are time consistent.  

  \begin{remark}
Notice that the risk-to-go satisfies the following equation:
\[
\risk_k(r^\ast_{k+1})+d_k(x_k,\pi_k^\ast(x_k,r^\ast_k))=r^\ast_k,\,\, \forall k\in \{0,\ldots, N-2\}.
\]
where $r_0^\ast=r_0$. Define
\[
\overline M_{k+1}=r^\ast_{k+1}+\sum_{j=0}^{k}d_j(x_j,\pi_j^\ast(x_j,r^\ast_j))
\]
We can show that $M_{N-1}$ satisfies the following risk sensitive Martingale property: $\risk_{k}(M_{k+1})=M_{k}$. 
\end{remark}

Next, we have the following corollary depicting the closed form solution policy of Problem $\mathcal{OPT}$. The proof is identical to the proof of Theorem \ref{them:optPoli} and will be omitted for brevity.
 \begin{corollary}\label{coro_TC_OPT}
Let $\pi \in \Pi$ be a policy recursively defined by the solution of Bellman's equation in (\ref{eq:T_L}):
\[
\pi_k(h_k) = u^*(x_k, r^\ast_k), \text{ with } r^\ast_k = r^\ast_{k-1}+L^\ast(x_{k-1},r^\ast_{k-1})(x_k),
\]
when $k\in \{1,\ldots, N-1\}$, and
\[
\pi_0(x_0) = u^*(x_0, r^\ast_0),
\]
where $r^\ast_0=r_0\in \Phi_0(x_0)$. Then, $\pi=\{\pi_0,\ldots,\pi_{N-1}\}$ is an optimal policy for Problem $\mathcal{OPT}$ with initial state $x_0$ and threshold $r_0$. Furthermore, the $k-$th subsequence of $\pi$, i.e., $\{\pi_k,\ldots,\pi_{N-1}\}$ is also an optimal history dependent policy to the tail subproblem of Problem $\mathcal{OPT}$.
 \end{corollary}
 Similar to problem $\mathcal{OPT}$, this corollary concludes that the optimal policy of problem $\mathcal{OPT}^M$ can also be constructed using both state update and the risk-to-go. The resultant policy is thus history dependent. However it is still unclear how one can obtain the analytical formula of the risk-to-go from merely solving the Bellman iteration in \eqref{eq:T_L}. 
  
\subsection{Analytical Formula for the Risk-to-go Update}
Based on the time consistency analysis in previous sections, we aim to derive the analytical update formula for the risk-to-go $r_k$. Before getting to the main result, we need the following technical lemmas.
\begin{lemma}\label{lem_phi_k}
Suppose $r_{k}=R_N^{\pi}(x_{k})-R_N^{\pi}(x_{k-1})+r_{k-1}$ for $k\in\{1,\ldots,N-1\}$ and $\pi\in\Pi$ is any admissible history dependent control policy. Then, if $r_0\in\Phi_0(x_0)$ such that $R_N^\pi(x_0)\leq r_0$, it implies $r_{k}\in\Phi_{k}(x_{k})$ for all $k\in\{0,\ldots,N-1\}$.
\end{lemma}
\begin{proof}
First, we characterize the lower bound of $r_{k}$. By definition, $r_{j+1}-r_j=R_N^{\pi}(x_{j+1})-R_N^{\pi}(x_{j})$, for $j\in\{0,\ldots,N-2\}$. By summing over $j=0$ to $k-1$, for any fixed $k\in\{0,\ldots,N-1\}$, we get
\[
\begin{split}
&r_{k}-r_0=R_N^\pi(x_{k})-R_N^\pi(x_{0})\\
\implies &r_{k}-R_N^\pi(x_{k})=r_0-R_N^\pi(x_{0})
\end{split}
\]
Thus, one obtains 
\[
r_k-\underline{R}_N(x_k)\geq r_k-R_N^\pi(x_{k})=r_0-R_N^\pi(x_{0})\geq 0.
\]
The last inequality is due to the fact that $\pi\in\Pi$ is a feasible control policy. Next, we characterize the upper bound for $r_{k}$. For $r_k=R_N^\pi(x_{k})-R_N^\pi(x_{k-1})+r_{k-1}$, by monotonicity and translation invariance of multi period risk measures, one can easily show that for any $k\in\{0,\ldots,N-1\}$,
\[
R_N^\pi(x_{k})\leq (N-k)\risk_{\text{max}},\,\,R_N^\pi(x_{k})\geq (N-k)\risk_{\text{min}}.
\]

Therefore, the above expressions imply
 \[
 r_k-r_{k-1}\leq (N-k)(\risk_{\text{max}}-\risk_{\text{min}})-\risk_{\text{min}}.
 \]
 By a telescopic sum, and since $r_0\in[\underline{R}_N(x_0),\infty)$, one obtains
  \[
  \begin{split}
 r_k\leq &\,\sum_{j=1}^k(N-j)(\risk_{\text{max}}-\risk_{\text{min}})-k\risk_{\text{min}}+r_0\\
 =&\,r_0+k\left(\left(N-\frac{k+1}{2}\right)(\risk_{\text{max}}-\risk_{\text{min}})-\risk_{\text{min}}\right)<\infty
 \end{split}
 \]
Thus, combining the result of lower bound for $r_k$, we get, $r_k\in[\underline{R}_N(x_k),\infty)$, which means $r_k\in\Phi_k(x_k)$ for $k\in\{0,\ldots,N-1\}$. 
\end{proof}
We are now in the position of deriving the main result of this paper. The following theorem provides an analytical update formula for the risk-to-go $r_k$.
\begin{theorem}\label{thm_TC_update}(Conceptual Risk-to-go)
Let $\pi^\ast\in\Pi$ be an optimal policy for Problem $\mathcal{OPT}$.  The following risk-to-go
\begin{equation}\label{r_defn_2}
\begin{split}
&\tilde{r}_k=r_0,\\
&\tilde{r}_{k+1}=\tilde{r}_k+R_N^{\pi^\ast}(x_{k+1})\!-\!R_N^{\pi^\ast}(x_{k}),\,\, \forall k\in\{0,\ldots,N-1\},
\end{split}
\end{equation}
and the augmented state-feedback control policy $\tilde\pi=\{\tilde\pi_0,\ldots,\tilde\pi_{N-1}\}$, where $\tilde\pi(x_k,\tilde r_k)=\pi^\ast(x_k)$ form time consistent solution to Problem $\mathcal{OPT}^M$.
 \end{theorem}
\begin{proof} 
Since $R_N^{\pi^\ast}(x_{k})$ and $R_N^{\pi^\ast}(x_{k+1})$ are formed by compounding Markov risk measures, we can easily see from Definition \ref{def:Markov} that  $R_N^{\pi^\ast}(x_{k})$ and $R_N^{\pi^\ast}(x_{k+1})$ are functions of $x_k$ and $x_{k+1}$. Also, define 
\[
\tilde{L}_k(x_{k+1})=R_N^{\pi^\ast}(x_{k+1})\!-\!R_N^{\pi^\ast}(x_{k}),
\]
for any $k\in\{0,\ldots,N-1\}$. From Lemma \ref{lem_phi_k}, one obtains $\tilde{r}_{k+1}=\tilde{r}_k+\tilde{L}_k(x_{k+1})\in\Phi_{k+1}(x_{k+1})$, $\forall k\in\{0,\ldots,N-2\}$. Also, as there is no terminal cost, by a telescopic sum,
\[
\tilde{r}_N=\tilde{r}_0+\sum_{k=0}^{N-1}\tilde{L}_k(x_{k+1})=-R_N^{\pi^\ast}(x_{0})+\tilde{r}_0.
\]
As $\pi^\ast$ is an optimal history dependent control policy for Problem $\mathcal{OPT}$, one obtains $R_N^{\pi^\ast}(x_{0})\leq \tilde{r}_0$ and $\tilde{r}_N\geq 0$. Furthermore, by the time consistent property and translational invariance of risk measures, we have that
\[
\begin{split}
&\risk_k(R_N^{\pi^\ast}(x_{k+1}))+c(x_k,\pi^\ast_k(x_k))=R_N^{\pi^\ast}(x_{k}),\,\, \forall k,\\
\iff&\risk_k(R_N^{\pi^\ast}(x_{k+1}))+c(x_k,\tilde\pi_k(x_k,r_k))=R_N^{\pi^\ast}(x_{k}),\,\, \forall k,
\end{split}
\]
Thus, the optimal policy $\tilde\pi$ and the risk-to-go sequence $\tilde r=\{\tilde r_1,\ldots,\tilde{r}_{N}\}$ are feasible to Problem $\mathcal{OPT}^M$. Then, $\forall k\in\{0,\ldots,N-1\}$,
\[
V_0(x_0,r_0)=J_N^{\pi^\ast}(x_0)=J_N^{\tilde\pi}(x_0)\geq V_0^{M}(x_0,r_0).
\]
At the same time, equation (\ref{eq_V_V_M}) implies that $V_0^{M}(x_0,r_0)\geq V_0(x_0,r_0)$. Thus both arguments imply $\tilde\pi$ and $\tilde r=\{\tilde r_1,\ldots,\tilde{r}_{N}\}$ form solution to Problem $\mathcal{OPT}^M$.  Time consistency then follows directly from Definition \ref{TC_defn}.
\end{proof}
From this theorem, we conclude that the analytical formula of the risk-to-go can be written as a Martingale difference of the constraint cost function. This property is crucial to understand how risk evaluation is updated at each step in order to make time consistent decisions and to derive large scale risk constrained decision making algorithms.  

\section{The Squander or Save Example (c.f. \cite{haviv1996constrained})}\label{sec:example}
We consider the simple case of risk constrained SOC problem where risk-neutral costs and risk neutral constraints are considered.  Given an MDP with initial state $x_0\in S_0$, risk threshold $r_0\in\Phi_0(x_0)$ and time horizon $N\geq 1$, solve
\begin{alignat*}{2}
\min_{\pi} & & \quad& \mathbb E\left[\sum_{k=0}^{N-1}\, c(x_k, u_k)\right] \\
\text{subject to} & & \quad&\mathbb E\left[\sum_{k=0}^{N-1}\, d(x_k, u_k) \right] \leq r_0,
\end{alignat*}

Consider the example in Figure \ref{fig:sm_act_1} and Figure \ref{fig:sm_act_2} with $N=2$, and all terminal costs equal to zero (they are not drawn in the figures). Also, let $S_0$ be the initial state space, 
\[
\begin{split}
&S_1=\{\text{``win a lottery"},\text{``lose a lottery"}\},\\
&U(\text{``win a lottery"})=\{\text{``squander"},\text{``save"}\},\\
&U(\text{``lose a lottery"})=\{\text{``squander"},\text{``save"}\}
\end{split}
\]
be the state and closed loop action space at stage 1, and 
\[
\begin{split}
S_2=\left\{\begin{array}{l}
\text{``win a lottery and squander"},\\
\text{``win a lottery and save"},\\
\text{``lose a lottery and squander"},\\
\text{``lose a lottery and save"}\end{array} \right\},\\
\end{split}
\]
be the state space at stage 2. Suppose there are no actions ($U(S_0)=\{\emptyset\},$), no cost ($c(x_0,u_0)=0$, $\forall (x_0,u_0)\in S_0\times U(S_0)$) and no constraint cost ($d(x_0,u_0)=0$, $\forall (x_0,u_0)\in S_0\times U(S_0)$) for winning/losing a lottery at stage 0. On the other hand, the stage-wise cost in stage 1 is as follows:
\[
\begin{split}
&c(`\text{`win a lottery"},u_1)=\left\{\begin{array}{ll}
-50&\text{if $u_1$=``squander"}\\
-30&\text{if $u_1$=``save"}
\end{array}\right.\\
&c(`\text{`lose a lottery"},u_1)=\left\{\begin{array}{ll}
-20&\text{if $u_1$=``squander"}\\
-10&\text{if $u_1$=``save"}
\end{array}\right.
\end{split}
\]
and the constraint stage-wise cost is as follows:
\[
\begin{split}
&d(`\text{`win a lottery"},u_1)=\left\{\begin{array}{ll}
1&\text{if $u_1$=``squander"}\\
0.05&\text{if $u_1$=``save"}
\end{array}\right.\\
&d(`\text{`lose a lottery"},u_1)=\left\{\begin{array}{ll}
0.4&\text{if $u_1$=``squander"}\\
0.2&\text{if $u_1$=``save"}
\end{array}\right.
\end{split}
\]

Suppose at time $0$, one has probability $\epsilon$ of winning a lottery and probability $1-\epsilon$ of losing a lottery where $0<\epsilon<<1$. If one wins the lottery, at stage 1, one can choose to squander (action 1) or to save (action 2). If one loses the lottery, at stage 1 one can choose to squander (action 1) or to save (action 2). In this example, the stage-wise cost represents a level of satisfaction in spending and it is inversely proportional to the money spent. On the other hand, the stage-wise constraint cost is the probability of going bankruptcy, which is directly proportional to the money spent. 

Consider $\epsilon=0.1$ and the risk threshold $r_0 = 0.3$. That is, the probability of winning lottery is $0.1$ and one wants to limit the probability of bankruptcy to be under $0.3$. Similar to Haviv's argument in \cite{haviv1996constrained}, if we keep the risk threshold constant at $0.3$ for all subsequent stages, the optimal policy decided at stage 0 is not to squander if one loses the lottery, and to squander if one wins the lottery. However, the optimal policy decided in stage 1 is to save if one loses or wins the lottery. This implies that the optimal control policies decided at stage 0 is time-inconsistent. 

On the other hand, suppose one finds the optimal control policies by solving  the Bellman's equation in (\ref{eq:T}). Then, from the value function for Problem $\mathcal{OPT}$, one obtains $r_2=0$ and $V(x_2,r_2)=0$. Now let $s_{11}=\text{`win a lottery"}$ and $s_{12}=\text{`lose a lottery"}$. At stage 1, the value function is as follows:
\[
\begin{split}
V_1(s_{11},r_1(s_{11}))=&\left\{\begin{array}{ll}
-50&\text{if $r_1(s_{11}) \geq 1$}\\
-30&\text{if $0.05\leq r_1(s_{11}) < 1$}\\
\infty&\text{otherwise}
\end{array}\right.\\
V_1(s_{12},r_1(s_{12}))=&\left\{\begin{array}{ll}
-20&\text{if $r_1(s_{12}) \geq 0.4$}\\
-10&\text{if $0.2\leq r_1(s_{12}) < 0.4$}\\
\infty&\text{otherwise}
\end{array}\right.
\end{split}
\]
and the optimal policy is as follows:
\[\small
\begin{split}
\pi^\ast_1(s_{11},r_1(s_{11}))=&\left\{\begin{array}{ll}
\text{``squander"}&\text{if $r_1(s_{11}) \geq 1$}\\
\text{``save"}&\text{if $0.05\leq r_1(s_{11}) < 1$}\\
\emptyset&\text{otherwise}
\end{array}\right.\\
\pi^\ast_1(s_{12},r_1(s_{12}))=&\left\{\begin{array}{ll}
\text{``squander"}&\text{if $r_1(s_{12}) \geq 0.4$}\\
\text{``save"}&\text{if $0.2\leq r_1(s_{12}) < 0.4$}\\
\emptyset&\text{otherwise}
\end{array}\right.
\end{split}
\]

At stage 0, with $r_0=0.3$, the value function is: 
\begin{alignat*}{2}
V_0(s_0,r_0)=\min_{r_1(\cdot)\in\reals^2}& &\quad&\!\!\!\!\! \left[0.1\times\left\{ \begin{array}{ll}
-50&\text{if $r_1(s_{11})\geq 1$}\\
-30&\text{if $0.05\leq r_1(s_{11})< 1$}\\
\infty&\text{if $ r_1(s_{11})<0.05$}\\
\end{array}\right.\right.\\
& &\quad &\!\!\!\!\!\!\!\!\!\!\!\!\!\!\!\left.+0.9\times\left\{ \begin{array}{ll}
-20&\text{if $r_1(s_{12})\geq 0.4$}\\
-10&\text{if $0.2\leq r_1(s_{12})< 0.4$}\\
\infty&\text{if $r_1(s_{12})< 0.2$}\\
\end{array}\right.\right]\\
\text{subject to}\!\!& &\quad& r_1(s_{11})\in[0.05,\infty],\,\, r_1(s_{12})\in [0.2,\infty],\\
& &\quad & 0.1\times r_1(s_{11})+0.9\times r_1(s_{12})\leq 0.3.
\end{alignat*}

This implies that the value function is 
\[
V_0(s_0,r_0)=-10(0.9)-30(0.1)=-12.
\]
Based on the risk-to-go update in equation (\ref{eq:T_L}), we have that
\[\small
\begin{split}
r^\ast(s_0,r_0)(s_{11})&=0.3+0.05-(0.1(0.05)+0.9(0.2))=0.165,\\
r^\ast(s_0,r_0)(s_{12})&=0.3+0.2-(0.1(0.05)+0.9(0.2))=0.315.
\end{split}
\]
Thus, the optimal history dependent policy at stage 0 is
\[
\begin{split}
\pi_0^\ast(h_0)&=\emptyset,\\
\pi_1^\ast(h_1)&=\left\{\begin{array}{ll}
\text{``save"}&\text{if $x_1=\text{`win a lottery"}$}\\
\text{``save"}&\text{if $x_1=\text{`lose a lottery"}$}\\
\end{array}\right.,
\end{split}
\]
i.e., this optimal policy is time consistent. 

\begin{figure}[htpb]
  \includegraphics[width = 0.5\textwidth]{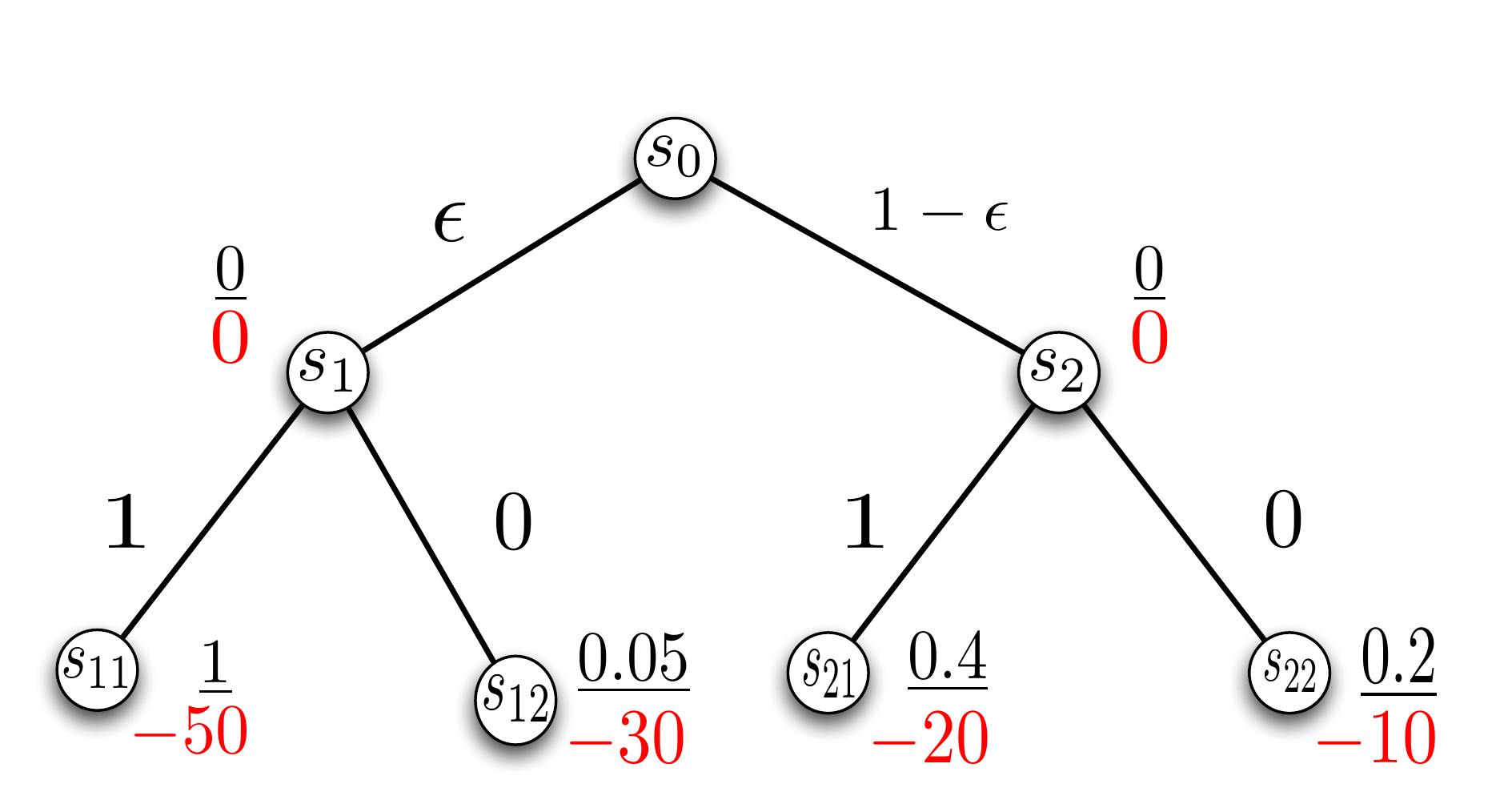}
  \caption{Transition probability, cost function and constraint cost function under policy 1 (squander). The red numbers indicates the stage-wise cost. The underlined black numbers indicates the stage-wise constraint cost}\label{fig:sm_act_1}
\end{figure}
\begin{figure}[htpb]
  \includegraphics[width = 0.5\textwidth]{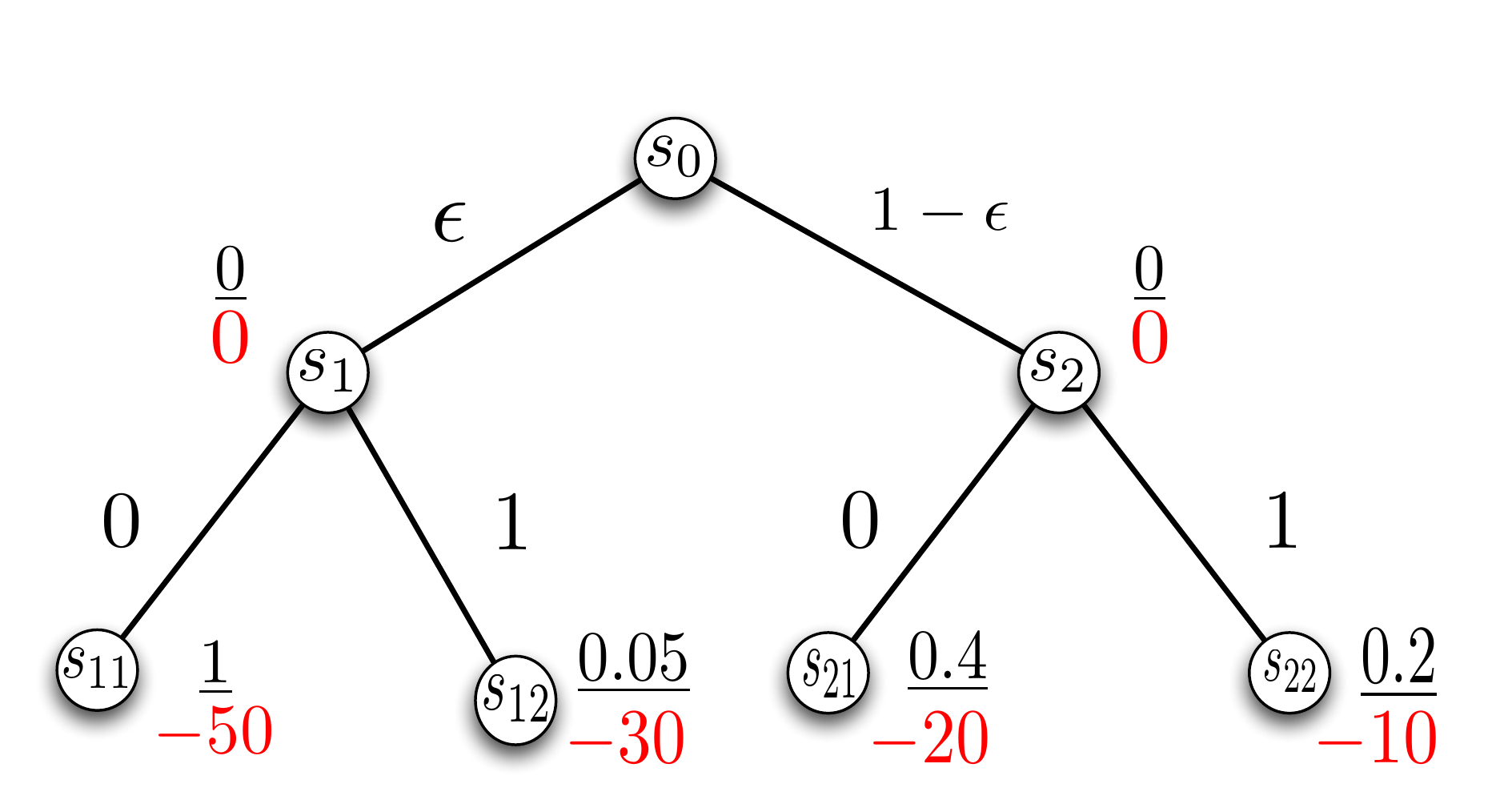}
 \caption{Transition probability, cost function and constraint cost function under policy 2 (save).  The red numbers indicates the stage-wise cost. The underlined black numbers indicates the stage-wise constraint cost.} \label{fig:sm_act_2}
\end{figure}

\section{Conclusion}\label{sec:conc}
In this paper we study time-consistency of risk constrained problem where the risk metric is time consistent. From the Bellman optimality condition in \cite{Chow_Pavone_13_1}, we establish an analytical ``risk-to-go" that results in a time consistent optimal policy.  The effectiveness of the analytical solution is also justified by solving Haviv's counter-example  \cite{haviv1996constrained} in time inconsistent planning. Future work includes extending the above analysis to large scale risk constrained decision making (via the use of approximate dynamic programming) that provides time consistent policies. 

\bibliographystyle{unsrt} 
\bibliography{ref_dyn_pro}  
\newpage
\appendix

\subsection{Proof of Theorem \ref{Bellman_M_thm}}
The proof follows from the Bellman's equation with risk constraints. First, we want to prove that 
\begin{equation}\label{eq_V_V_M}
V^M_{k}(x_k, r_k)=V_{k}(x_k, r_k)
\end{equation}
 for any $k\in\{0,\ldots,N-1\}$, $x_k\in S$ and $r_k\in\Phi_k(x_k)$ and $k\in\{0,\ldots,N-1\}$. At $k=N-1$, by definition, $V^M_{N-1}(x_{N-1}, r_{N-1})=\overline{T}_{N-1}[V_N^M](x_{N-1},r_{N-1})=\overline{T}_{N-1}[V_N](x_{N-1},r_{N-1})$. Since $V_{N}(x_N, r_N)=V^M_{N}(x_N, r_N)=0$ for any $r_N\in[0,\infty]$, by Lemma \ref{tech_lem_F}, one obtains $V^M_{N-1}(x_{N-1}, r_{N-1})=\overline{T}_{N-1}[V_N](x_{N-1},r_{N-1})=T_{N-1}[V_N](x_{N-1},r_{N-1})=V_{N-1}(x_{N-1}, r_{N-1})$. Thus, equation (\ref{eq_V_V_M}) holds for $k=N-1$. By inductive hypothesis, suppose for $k=j+1$, 
\[
V^{M}_{j+1}(x_{j+1},r_{j+1}) =V_{j+1}(x_{j+1},r_{j+1})
\]
Then for $k=j$. Let $\pi^* \in \Pi^{M}_j$ be the optimal policy that yields the optimal cost $V^{M}_{j}(x_j, r_j)$ and $r^\ast$ be the sequence of optimal risk threshold update functions. By applying the law of total expectation, we can write:
\begin{equation*}
\begin{split}
V^{M}_{j}(x_j, r_j) &= \expectation{\sum_{i=j}^{N-1} \, c_i(x_i, \pi^*_i(x_i,r_i))} \\
&=  \expectation{c_j(x_j, \pi_j^*(x_j,r_j))+   J_N^{ \pi^\ast}(x_{j+1}) }.
\end{split}
\end{equation*}
Clearly, $\pi^\ast$ is a feasible policy for the $j+1^{\text{th}}$ tail subproblem of Problem $\mathcal{OPT}^\textrm{M}$ with  $x_{j+1}\in S$ and $r^\ast_{j+1}\in\Phi_{j+1}(x_{j+1})$. Collecting the above results, we can write 
\begin{equation*}
\begin{split}
V^{M}_{j}&(x_j, r_j)  \geq   \expectation{c_j(x_j, \pi_j^*(x_j,r_j)) + V^{M}_{j+1}(x_{j+1}, r^\ast_{j+1})}\\
&\geq \overline{T}_j[V^{M}_{j+1}](x_j, r_j)= \overline{T}_j[V_{j+1}](x_j, r_j)\\
&= {T}_j[V_{j+1}](x_j, r_j)=V_j(x_j,r_j).
\end{split}
\end{equation*}
The second inequality follows from the fact that 
\[
\risk_j(r_{j+1}^\ast)+d_j(x_j,\pi^\ast_j(x_j,r_j)))\leq r_j.
\]
and the first equality follows from induction's assumption.

On the other hand, for given pair $(x_j, r_j)$, where $r_j \in \Phi_j(x_j)$, let $u^*(x_j,r_j)$ and $r^{\prime, *}(x_{j},r_{j})(x_{j+1})$ be the minimizers in $\overline{T}_j[V^{M}_{j+1}](x_j, r_j)$. Construct  a policy $\bar \pi \in \Pi^{M}_j$ as follows: $\bar \pi_j(x_j,r_j) = u^*(x_j,r_j)$ and $\bar \pi_i(x_i,r_i) = \pi^*_i(x_i,r_i)$ for $i\geq j+1$. Therefore, the policy $\bar \pi \in \Pi^M_j$ is a feasible policy for the $j+1^{\text{th}}$ tail subproblem of Problem $\mathcal{OPT}^M$ with $r_j\in\Phi_j(x_j)$. 
Hence, one easily obtains:
\begin{equation*}
\begin{split}
V^{M}_{j}&(x_j, r_j)\leq J^{\bar \pi}_N(x_j)\!=\!  \expectation{c_j(x_j, \bar \pi_j(x_j,r_j)))+ J^{\pi^*}_N(x_{j+1})}\\
 =&\overline{T}_j[V^M_{j+1}](x_j, r_j)=\overline{T}_j[V_{j+1}](x_j, r_j)=V_j(x_j,r_j). 
\end{split}
\end{equation*}

By combining both steps, the claim in equation (\ref{eq_V_V_M}) is proved by induction. Furthermore, by the equation (\ref{eq_V_V_M}) and theorem \ref{TC_good}, the following expressions hold:
\begin{equation}\label{Bellman_M}
\begin{split}
V^{M}_{k}(x_{k},r_{k})& =V_{k}(x_{k},r_{k})=\overline{T}_k[V_{k+1}](x_k,r_k)\\
&=\overline{T}_k[V^M_{k+1}](x_k,r_k).
\end{split}
\end{equation}
By repeatedly applying this Bellman's equation, one obtains
\begin{equation}\label{Bellman_M_seq}
V^M_{0}(x_0, r_0) = \overline{T}_0[\overline{T}_{1}[\ldots[\overline{T}_{N-1}[V^M_{N}]]\ldots]](x_0,r_0),
\end{equation}
Also, define the sequence of optimal policy: $\pi^\ast=\{\pi_0^\ast,\ldots,\pi_{N-1}^\ast\}$ and sequence of risk-to-go: $r^\ast=\{r_1^\ast,\ldots,r_{N}^\ast\}$ found by solving the sequence of Bellman's recursions in equation (\ref{Bellman_M_seq}). Now, consider the k subsequence: $\tilde\pi^\ast=\{\pi^\ast_{k},\ldots,\pi^\ast_{N-1}\}\in\Pi^M_k$ of $\pi^\ast$ and the k subsequence: $\tilde{r}^\ast=\{r^\ast_{k+1},\ldots,r^\ast_{N}\}$. By definition, one obtains,
\[
d(x_j, \pi_j^\ast(x_j,r_j)) + \risk_j(r^\ast_{j+1})  \leq r_j,\,\, \forall N-1\geq j\geq k.
\]
and the Bellman's equation implies
\[
J^{\tilde{\pi}^\ast}_N(x_k)\!=\! \overline{T}_k[\ldots[\overline{T}_{N-1}[V^M_{N}]]\ldots](x_k,r_k)\!=\!V^M_k(x_k,r_k).
\]
This implies that $\tilde\pi^\ast$ is a sequence of optimal control policy of the k-tail subproblem of Problem $\mathcal{OPT}^M$, for any state $x_k\in S$, and any risk threshold $r_{k}\in\Phi_k(x_k)$. 

\addtolength{\textheight}{-14cm}   
\end{document}